\documentclass[11pt]{article}%
\usepackage[letterpaper, portrait, margin=1in, headheight=0pt]{geometry}
\usepackage{mathtools}
\usepackage{enumerate}
\usepackage{amsmath,amsthm,amssymb,amsfonts}
\usepackage{enumitem}
\usepackage{mathrsfs}
\usepackage{tikz-cd}
\usepackage{multicol}
\usepackage[title]{appendix}
\usetikzlibrary{quotes,angles}

\DeclareMathOperator{\Hom}{Hom}

\DeclareMathOperator{\Inf}{Inf}
\DeclareMathOperator{\Tra}{Tra}
\DeclareMathOperator{\Res}{Res}
\DeclareMathOperator{\ima}{Im}

\newcommand{\F}{\mathbb{F}}

\newcommand{\inv}{^{-1}}

\newcommand{\cH}{\mathcal{H}}

\newcommand{\vv}{^{\vdash}}
\newcommand{\dd}{^{\dashv}}
\newcommand{\FR}{F\lozenge R + R\lozenge F}

\newtheorem{thm}{Theorem}[section]
\newtheorem{lem}[thm]{Lemma}

\newtheorem{cor}[thm]{Corollary}

\theoremstyle{definition}

\theoremstyle{remark}

\title{Multipliers and Covers of Perfect Diassociative Algebras}
\author{Erik Mainellis}
\date{}

\begin{document}

\maketitle

\begin{abstract}
    The paper concerns perfect diassociative algebras and their implications to the theory of central extensions. It is first established that perfect diassociative algebras have strong ties with universal central extensions. Then, using a known characterization of the multiplier in terms of a free presentation, we obtain a special cover for perfect diassociative algebras, as well as some of its properties. The subsequent results connect and build on the previous topics. For the final theorem, we invoke an extended Hochschild-Serre type spectral sequence to show that, for a perfect diassociative algebra, its cover is perfect and has trivial multiplier. This paper is part of an ongoing project to advance extension theory in the context of several Loday algebras.
\end{abstract}

\section{Introduction}
In 2001, Loday introduced associative dialgebras, or diassociative algebras, in the context of algebraic $K$-theory \cite{loday dialgebras}. They have since been found to have connections with algebraic topology, among other fields, and nilpotent dialgebras have been classified up to dimension 4 (see \cite{basri}). Explicitly, a diassociative algebra $L$ consists of a vector space equipped with two associative bilinear products, denoted $\dashv$ and $\vdash$, that satisfy
\begin{align}
    x\dashv (y\dashv z) = x\dashv (y\vdash z),\\
    (x\vdash y)\dashv z = x\vdash (y\dashv z),\\
    (x\dashv y)\vdash z = (x\vdash y)\vdash z.
\end{align}
for all $x,y,z\in L$.

Over the last century, a large volume of group-theoretic notions have been established in the context of Lie algebras, and many of these have also been generalized to Leibniz or other classes of algebras. In particular, the Schur multiplier originated in the context of group representation theory (see \cite{karp}), but has since been studied for Lie algebras \cite{batten}, for Leibniz algebras \cite{mainellis batten, rogers}, and for diassociative algebras \cite{mainellis batten di}. These works contain an extension-theoretic investigation of multipliers and covers, low-dimensional cohomology, and unicentral algebras. Although diassociative algebras are quite different than Lie algebras, many techniques used on Lie algebras have found success in the relatively new context of the former. In the present paper, we are interested in the subclass of perfect algebras, i.e. algebras that are equal to their derived ideal. The multipliers, covers, and universal central extensions related to perfect Lie algebras are known to have remarkable properties \cite{batten}. These, too, take their motivations from group theory \cite{karp}. Some were generalized to Hom-Leibniz algebras in \cite{casas}. The objective of the present paper is to investigate the same matters for diassociative algebras.

We use the following definitions throughout. Given a diassociative algebra $L$ with multiplications $\dashv$ and $\vdash$, consider two subalgebras $S$ and $T$ in $L$. We denote $S\lozenge T = S\dashv T + S\vdash T$ and say that $L$ is \textit{perfect} if $L= L'$, where $L' = L^2 = L\lozenge L$. Given another diassociative algebra $A$, an \textit{extension} of $A$ by $L$ is a short exact sequence $0\xrightarrow{} A\xrightarrow{\sigma} H\xrightarrow{\pi} L\xrightarrow{} 0$ for which $\sigma$ and $\pi$ are homomorphisms and $H$ is a diassociative algebra. In general, one may assume that $\sigma$ is the identity map, and we make this assumption throughout the paper. An extension is called \textit{central} if $A\subseteq Z(H)$. A \textit{section} of an extension is a linear map $\mu:L\xrightarrow{} H$ such that $\pi\circ \mu = \text{id}_L$.

The paper is structured as follows. In Section 2, we establish a series of lemmas that relate universal central extensions to perfect algebras. One such result states that, for any universal central extension $0\xrightarrow{} A\xrightarrow{} H\xrightarrow{} L\xrightarrow{} 0$, both $L$ and $H$ are perfect (Lemma \ref{batten 6.2}). Perhaps most notably, it is shown that, given a perfect diassociative algebra $L$, the extension $0\xrightarrow{} 0\xrightarrow{} L\xrightarrow{} L\xrightarrow{} 0$ is universal if and only if every central extension of $L$ splits. In Section 3, we study the multipliers and covers of finite-dimensional perfect diassociative algebras. Given such an $L$, and using the characterization of $M(L)$ in terms of a free presentation $0\xrightarrow{} R\xrightarrow{} F\xrightarrow{} L\xrightarrow{} 0$ of $L$ (as established in \cite{mainellis batten di}), we prove that $F'/(\FR)$ is a cover of $L$. In particular, and among other statements, it is also shown that the extension \[0\xrightarrow{} M(L)\cong \frac{F'\cap R}{\FR}\xrightarrow{} \frac{F'}{\FR}\xrightarrow{} L\xrightarrow{} 0\] is universal. Next, given a universal central extension $0\xrightarrow{} A\xrightarrow{} L^*\xrightarrow{} L\xrightarrow{} 0$ of a perfect algebra $L$, we obtain that $A\cong M(L)$ and that $L^*$ is a cover of $L$. We also explore what happens when $L$ has trivial multiplier. Finally, we use the extended low-dimensional Hochschild-Serre type spectral sequence, as established in \cite{mainellis batten di}, to prove that $C=C'$ and $M(C)=0$ for any cover $C$ of a finite-dimensional perfect diassociative algebra $L$.

\section{Universal Central Extensions}
The aim of this section is establish connections between perfect diassociative algebras and universal central extensions. Consider a finite-dimensional diassociative algebra $L$, an $L$-module $A$, and two central extensions $E:0\xrightarrow{} A\xrightarrow{} H\xrightarrow{} L\xrightarrow{} 0$ and $E_1:0\xrightarrow{} A_1\xrightarrow{} H_1\xrightarrow{} L\xrightarrow{} 0$. We say that $E$ \textit{covers} $E_1$ if there exists a homomorphism $\tau:H\xrightarrow{} H_1$ such that the diagram \[\begin{tikzcd}
E:0\arrow[r] &A\arrow[r] \arrow[,d] & H \arrow[d,"\tau"] \arrow[r] & L\arrow[r] \arrow["\text{id}",d]&0\\
E_1:0\arrow[r] &A_1\arrow[r,swap] & H_1 \arrow[r,swap] &L\arrow[r] &0
\end{tikzcd}\] commutes, where the unmarked map is $\tau|_A$. If $\tau$ is unique, then $E$ \textit{uniquely covers} $E_1$. An extension $E$ is \textit{universal} if it uniquely covers any central extension of $L$.

\begin{lem}\label{batten 6.1}
If $E:0\xrightarrow{} A\xrightarrow{} H\xrightarrow{\phi} L\xrightarrow{} 0$ and $E_1:0\xrightarrow{} A_1\xrightarrow{} H_1\xrightarrow{\phi_1} L\xrightarrow{} 0$ are universal central extensions of $L$, then there exists an isomorphism $H\xrightarrow{} H_1$ which carries $A$ onto $A_1$.
\end{lem}

\begin{proof}
Since both extensions are universal and central, there exist homomorphisms $\tau:H\xrightarrow{} H_1$ and $\tau_1:H_1\xrightarrow{} H$ such that the diagrams \[\begin{tikzcd}
0\arrow[r] &A\arrow[r] \arrow[d] & H \arrow[d,swap,"\tau"] \arrow[r,"\phi"] & L\arrow[r] \arrow["\text{id}",d]&0\\
0\arrow[r] &A_1\arrow[r,swap] & H_1 \arrow[r,swap,"\phi_1"] &L\arrow[r] &0
\end{tikzcd} ~~~~~~~~~~ \begin{tikzcd}
0\arrow[r] &A_1\arrow[r] \arrow[d] & H_1 \arrow[d,swap,"\tau_1"] \arrow[r,"\phi_1"] & L\arrow[r] \arrow["\text{id}",d]&0\\
0\arrow[r] & A\arrow[r,swap] & H \arrow[r,swap, "\phi"] &L\arrow[r] &0
\end{tikzcd}\] commute. The mapping $\tau_1\circ \tau:H\xrightarrow{} H$ is then such that $\phi\circ \tau_1\circ \tau = \phi_1\circ \tau = \phi$. Since $\tau$ is unique, $\tau_1\circ \tau = \text{id}_H$. Similarly, $\tau\circ \tau_1 = \text{id}_{H_1}$. Therefore $\tau$ is an isomorphism $H\xrightarrow{} H_1$ and $\tau|_A$ maps $A$ onto $A_1$.
\end{proof}

\begin{lem}\label{batten 6.2}
If $E:0\xrightarrow{} A\xrightarrow{} H\xrightarrow{\phi} L\xrightarrow{} 0$ is a universal central extension, then both $H$ and $L$ are perfect.
\end{lem}

\begin{proof}
Consider the central extension $0\xrightarrow{} A\times H/H'\xrightarrow{} H\times H/H'\xrightarrow{\psi} L\xrightarrow{} 0$ where $\psi(a,b) = \phi(a)$ for $a\in H$ and $b\in H/H'$. Define homomorphisms $\tau_i:H\xrightarrow{} H\times H/H'$ for $i=1,2$ by $\tau_1(h) = (h,0)$ and $\tau_2(h) = (h,h+H/H')$. Then $\psi\circ \tau_1(h) = \psi(h,0) = \phi(h)$ and $\psi\circ \tau_2(h) = \psi(h,h+H/H') = \phi(h)$, which implies that $\psi\circ \tau_i = \phi$ for $i=1,2$. Since $E$ is universal, we have $\tau_1=\tau_2$. Thus $H/H' = 0$ and $H=H'$. One computes \[L' = (H/A)' = \frac{H'+A}{A} = \frac{H+A}{A} = H/A = L\] and therefore $H$ and $L$ are perfect.
\end{proof}

\begin{lem}\label{batten 6.3}
Let $E:0\xrightarrow{} A\xrightarrow{} H\xrightarrow{\phi} L\xrightarrow{} 0$ and $E_1:0\xrightarrow{} A_1\xrightarrow{} H_1\xrightarrow{\phi_1} L\xrightarrow{} 0$ be central extensions and suppose $H$ is perfect. Then $E$ covers $E_1$ if and only if $E$ uniquely covers $E_1$.
\end{lem}

\begin{proof}
The reverse direction is clear. In the forward direction, suppose $E$ covers $E_1$. Then there exists a homomorphism $\tau:H\xrightarrow{} H_1$ such that the diagram \[\begin{tikzcd}
0\arrow[r] &A\arrow[r] \arrow[,d] & H \arrow[d,swap,"\tau"] \arrow[r,"\phi"] & L\arrow[r] \arrow["\text{id}",d]&0\\
0\arrow[r] &A_1\arrow[r,swap] & H_1 \arrow[r,swap,"\phi_1"] &L\arrow[r] &0
\end{tikzcd}\] commutes. Suppose there is another homomorphism $\beta:H\xrightarrow{} H_1$ such that the diagram \[\begin{tikzcd}
0\arrow[r] &A\arrow[r] \arrow[,d] & H \arrow[d,swap,"\beta"] \arrow[r,"\phi"] & L\arrow[r] \arrow["\text{id}",d]&0\\
0\arrow[r] &A_1\arrow[r,swap] & H_1 \arrow[r,swap,"\phi_1"] &L\arrow[r] &0
\end{tikzcd}\] commutes. It remains to show that $\tau = \beta$. Let $x,y\in H$. Then \begin{align*}
    \phi_1(\beta(x) - \tau(x)) &= \phi_1(\beta(x)) - \phi_1(\tau(x)) \\ &= \phi(x) - \phi(x) = 0
\end{align*} implies that $\beta(x) - \tau(x) \in \ker \phi_1 = A_1 \subseteq Z(H_1)$. Similarly, one obtains $\beta(y) - \tau(y)\in Z(H_1)$, and so $\beta(x) = \tau(x) + a$ and $\beta(y) = \tau(y) + b$ for some $a,b\in Z(H_1)$. We compute \begin{align*}
    \beta(x\dashv y) &= \beta(x)\dashv \beta(y) \\ &= \tau(x)\dashv \tau(y) + \tau(x)\dashv b + a\dashv \tau(y) + a\dashv b \\ &= \tau(x\dashv y)
\end{align*} since $a,b\in Z(H_1)$. Similarly, $\beta(x\vdash y) = \tau(x\vdash y)$, and thus $\tau$ and $\beta$ are equal on $H'$. Since $H$ is perfect, we have $\tau = \beta$.
\end{proof}

An extension $0\xrightarrow{} A\xrightarrow{} H\xrightarrow{\phi} L\xrightarrow{} 0$ \textit{splits} if there is a homomorphism $\mu:L\xrightarrow{} H$ that is also a section, i.e. that satisfies $\phi\circ \mu = \text{id}_L$.

\begin{lem}\label{batten 6.4}
Let $L$ be a finite-dimensional perfect diassociative algebra. Then the extension $E:0\xrightarrow{} 0\xrightarrow{} L\xrightarrow{} L\xrightarrow{} 0$ is universal if and only if every central extension of $L$ splits.
\end{lem}

\begin{proof}
In the forward direction, let $E_1:0\xrightarrow{} A\xrightarrow{} H\xrightarrow{\phi} L\xrightarrow{} 0$ be a central extension of $L$. Then there exists a unique homomorphism $\tau:L\xrightarrow{} H$ such that the diagram \[\begin{tikzcd}
E:0\arrow[r] &0\arrow[r] \arrow[,d] & L \arrow[d,swap,"\tau"] \arrow[r,"\text{id}"] & L\arrow[r] \arrow["\text{id}",d]&0\\
E_1:0\arrow[r] &A\arrow[r,swap] & H \arrow[r,swap,"\phi"] &L\arrow[r] &0
\end{tikzcd}\] commutes. Therefore $\phi\circ \tau = \text{id}_L$, which implies that $E_1$ splits.

Conversely, suppose every central extension of $L$ splits and let $E_1:0\xrightarrow{} A\xrightarrow{} H\xrightarrow{\phi} L\xrightarrow{} 0$ be a central extension. Then there exists a homomorphism $\beta:L\xrightarrow{} H$ such that $\phi\circ \beta = \text{id}_L$, which implies that the diagram \[\begin{tikzcd}
E:0\arrow[r] &0\arrow[r] \arrow[,d] & L \arrow[d,swap,"\beta"] \arrow[r,"\text{id}"] & L\arrow[r] \arrow["\text{id}",d]&0\\
E_1:0\arrow[r] &A\arrow[r,swap] & H \arrow[r,swap,"\phi"] &L\arrow[r] &0
\end{tikzcd}\] commutes. Thus $E$ covers $E_1$. Since $L$ is perfect, Lemma \ref{batten 6.3} guarantees that $E$ uniquely covers $E_1$, and therefore $E$ is universal.
\end{proof}

\begin{lem}\label{batten 6.6}
Let $E_1:0\xrightarrow{} B\xrightarrow{} G\xrightarrow{\phi} L\xrightarrow{} 0$ and $E_2:0\xrightarrow{} C\xrightarrow{} L\xrightarrow{\psi} H\xrightarrow{} 0$ be central extensions and let $\pi = \psi\circ \phi$ and $A=\ker \pi$. If $G$ is perfect, then $E_3:0\xrightarrow{} A\xrightarrow{} G\xrightarrow{\pi} H\xrightarrow{} 0$ is a central extension.
\end{lem}

\[\begin{tikzcd}
E_3 &&& E_1 \\ &0\arrow[rd] & & 0 \arrow[d] && \\ & &A \arrow[rd] & B \arrow[d] & \\ &&& G\arrow[rd, "\pi"] \arrow[d, swap, "\phi"] & \\ E_2 &0\arrow[r] &C \arrow[r] &L\arrow[r, swap, "\psi"]\arrow[d] & H\arrow[r]\arrow[rd] &0 \\ &&& 0 & & 0
\end{tikzcd}\]

\begin{proof}
For any $a\in A = \ker \pi$, one has $\psi\circ \phi(a) = \pi(a) =0$. Therefore, for any $x\in G$, $\phi(a\dashv x) = \phi(a)\dashv \phi(x) = 0$ since $\phi(a)\in \ker \psi$. Similarly, $\phi(a\dashv x)$, $\phi(x\dashv a)$, and $\phi(x\vdash a)$ are all zero. For $a\in A$, let $\lambda_a\dd$, $\lambda_a\vv$, $\rho_a\dd$, $\rho_a\vv$ be adjoint operators on $G$ defined by $\lambda_a^*(x) = a* x$ and $\rho_a^*(x) = x* a$, where $*$ ranges over $\dashv$ and $\vdash$. Since $\phi(\lambda_a^*(x)) = 0$ and $\phi(\rho_a^*(x)) = 0$, we get $\lambda_a^*(x),\rho_a^*(x)\in \ker \phi \subseteq Z(G)$. Now let $y,z\in G$. Then \begin{align*}
    \lambda_a\dd(y\dashv z) = a\dashv (y\dashv z) \overset{as}{=} (a\dashv y)\dashv z = 0,\\ \lambda_a\dd(y\vdash z) = a\dashv (y\vdash z) \overset{(1)}{=} a\dashv(y\dashv z) \overset{as}{=} (a\dashv y)\dashv z = 0,\\ \lambda_a\vv(y\dashv z) = a\vdash (y\dashv z) \overset{(2)}{=} (a\vdash y)\dashv z = 0,\\ \lambda_a\vv(y\vdash z) = a\vdash (y\vdash z) \overset{as}{=} (a\vdash y)\vdash z = 0,\\ ~ \\ \rho_a\dd(y\dashv z) = (y\dashv z)\dashv a \overset{as}{=} y\dashv(z\dashv a) = 0,\\ \rho_a\dd(y\vdash z) = (y\vdash z)\dashv a \overset{(2)}{=} y\vdash(z\dashv a) = 0,\\ \rho_a\vv(y\dashv z) = (y\dashv z)\vdash a \overset{(3)}{=} (y\vdash z)\vdash a \overset{as}{=} y\vdash(z\vdash a) = 0,\\ \rho_a\vv(y\vdash z) = (y\vdash z)\vdash a \overset{as}{=} y\vdash(z\vdash a) = 0
\end{align*} since $a\dashv y, a\vdash y, z\dashv a, z\vdash a\in Z(G)$. Therefore $\lambda_a^*$ and $\rho_a^*$ are trivial maps on $G'=G$, and so $a\in Z(G)$, which implies that $E_3$ is central.
\end{proof}

\begin{lem}
Let $E_1$, $E_2$, $E_3$, and involved maps be as in Lemma \ref{batten 6.6}. If $E_1$ is universal, then so is $E_3$.
\end{lem}

\begin{proof}
Suppose $E_1:0\xrightarrow{} B\xrightarrow{} G\xrightarrow{\phi} L\xrightarrow{} 0$ is a universal extension. By Lemma \ref{batten 6.2}, $G$ and $L$ are perfect. Since $H$ is the homomorphic image of $G$, $H$ is also perfect. Let $E_4:0\xrightarrow{} D\xrightarrow{} S\xrightarrow{\mu} H\xrightarrow{} 0$ be another central extension of $H$. Let $T=\{(a,b)\in L\times S~|~ \psi(a) = \mu(b)\}$ and define multiplications on $T$ by $(a,b)\dashv(c,d) = (a\dashv c,b\dashv d)$ and $(a,b)\vdash (c,d) = (a\vdash c,b\vdash d)$. Then $T$ is closed under multiplication since $\psi(a\dashv c) = \psi(a)\dashv\psi(c) = \mu(b)\dashv\mu(d) = \mu(b\dashv d)$ and $\psi(a\vdash c) = \mu(b\vdash d)$ similarly. Thus $T$ is a subalgebra of $L\times S$. Let $\lambda$ be the projection of $T$ onto $L$. Since $E_1$ is universal, there exists a unique homomorphism $\alpha:G\xrightarrow{} T$ such that the diagram \[\begin{tikzcd}
E_1:0\arrow[r] &B\arrow[r] \arrow[,d] & G \arrow[d,swap,"\alpha"] \arrow[r,"\phi"] & L\arrow[r] \arrow["\text{id}",d]&0\\
0\arrow[r] &0\times D\arrow[r,swap] & T \arrow[r,swap,"\lambda"] &L\arrow[r] &0
\end{tikzcd}\] commutes, i.e. $\lambda\circ \alpha = \phi$. Let $\gamma:T\xrightarrow{} S$ be the natural projection given by $\lambda(a,b) = b$. Let $\beta=\gamma\circ\alpha$. Given $g\in G$, set $\alpha(g) = (a,b)$. Then $\beta(g) = \gamma\circ\alpha(g) = \gamma(a,b) = b$ and $\phi(g) = \lambda\circ \alpha(g) = \lambda(a,b) = a$, which implies that $(\mu\circ \beta)(g) = \mu(b) = \psi(a) = \psi\circ\phi(g) = \pi(g)$. Thus the diagram \[\begin{tikzcd}
E_3:0\arrow[r] &A\arrow[r] \arrow[,d] & G \arrow[d,swap,"\beta"] \arrow[r,"\pi"] & H\arrow[r] \arrow["\text{id}",d]&0\\
E_4:0\arrow[r] & D\arrow[r,swap] & S \arrow[r,swap,"\mu"] &H\arrow[r] &0
\end{tikzcd}\] commutes, and so $E_3$ covers $E_4$. Since $G$ is perfect, we know that $E_3$ uniquely covers $E_4$. Therefore, $E_3$ is universal.
\end{proof}

The proof of the following lemma is similar to the Lie case (Lemma 6.7 in \cite{batten}), but we provide it here for the sake of completeness and to detail how it fits in with the previous lemmas.

\begin{lem}\label{batten 6.7}
Let $L$ be a finite-dimensional perfect diassociative algebra. If $0\xrightarrow{} 0\xrightarrow{} H\xrightarrow{\phi} L\xrightarrow{} 0$ is a universal extension, then so is $0\xrightarrow{} 0\xrightarrow{} L\xrightarrow{\emph{id}} L\xrightarrow{} 0$.
\end{lem}

\begin{proof}
Let $0\xrightarrow{} A\xrightarrow{} L^*\xrightarrow{\psi} L\xrightarrow{} 0$ be an arbitrary central extension of $L$. Then there exists a unique homomorphism $\theta:H\xrightarrow{} L^*$ such that the diagram \[\begin{tikzcd}
0\arrow[r] &0\arrow[r] \arrow[,d] & H \arrow[d,swap,"\theta"] \arrow[r,"\phi"] & L\arrow[r] \arrow["\text{id}",d]&0\\
0\arrow[r] &A\arrow[r,swap] & L^* \arrow[r,swap,"\psi"] &L\arrow[r] &0
\end{tikzcd}\] commutes. In other words, $\phi = \psi\circ \theta$. However, since $\phi$ is an isomorphism, let $\beta = \theta\circ \phi\inv$ be the homomorphism from $L$ to $L^*$. One computes $\psi\circ \beta = \psi\circ\theta\circ\phi\inv = \phi\circ \phi\inv = \text{id}_L$, which means $\beta$ is a section of $\psi$ that is also a homomorphism. Thus, our arbitrary central extension $0\xrightarrow{} A\xrightarrow{} L^*\xrightarrow{\psi} L\xrightarrow{} 0$ splits. By Lemma \ref{batten 6.4}, and since $L$ is perfect, we know that $0\xrightarrow{} 0\xrightarrow{} L\xrightarrow{} L\xrightarrow{} 0$ is universal.
\end{proof}

\section{Multipliers and Covers}
We now examine covers and multipliers of perfect diassociative algebras. Recall that a \textit{definining pair} $(K,M)$ of a diassociative algebra $L$ is itself a pair of diassociative algebras that satisfies $K/M\cong L$ and $M\subseteq Z(K)\cap K'$. Such a pair is called a \textit{maximal defining pair} if the dimension of $K$ is maximal. In this case, we say that $K$ is a \textit{cover} of $L$ and that $M$ is the \textit{multiplier} of $L$, denoted by $M(L)$. In \cite{mainellis batten di}, it is shown that \[M(L) = \frac{F'\cap R}{\FR}\] where $0\xrightarrow{} R\xrightarrow{} F\xrightarrow{} L\xrightarrow{} 0$ is a free presentation of $L$. In the same paper, it is also shown that $M(L)\cong \cH^2(L,\F)$.

Let $L$ be a perfect diassociative algebra with free presentation $0\xrightarrow{} R\xrightarrow{} F\xrightarrow{} L\xrightarrow{} 0$ and consider the natural extension \[0\xrightarrow{} \frac{R}{\FR}\xrightarrow{} \frac{F}{\FR} \xrightarrow{\pi} L\xrightarrow{} 0.\] Since $L'$ is perfect, we compute \begin{align*}
    L' &= \ima\Big(\frac{F}{\FR}\Big)' \\ &\cong \ima\Big(\frac{F' + \FR}{\FR} \Big)\\ &= \ima\Big(\frac{F'}{\FR}\Big)
\end{align*} which implies that the restriction $\pi|_{F'/(\FR)}$ induces a central extension \[0\xrightarrow{} \frac{F'\cap R}{\FR} \xrightarrow{} \frac{F'}{\FR} \xrightarrow{} L\xrightarrow{} 0.\] We know that the algebra $(F'\cap R)/(\FR)$ is precisely the multiplier $M(L)$. Our goal is now to show that $F'/(\FR)$ is a cover of $L$ and that the above extension is universal.

\begin{thm}\label{batten 6.8}
Let $L$ be a finite-dimensional perfect diassociative algebra and $0\xrightarrow{} R\xrightarrow{} F\xrightarrow{} L\xrightarrow{} 0$ be a free presentation of $L$. Then $F'/(\FR)$ is a cover of $L$.
\end{thm}

\begin{proof}
Since $M(L)\cong (F'\cap R)/(\FR)$ and \[L\cong \frac{F'/(\FR)}{(F'\cap R)/(\FR)},\] it remains to prove that \[\frac{F'\cap R}{\FR} \subseteq Z\Big(\frac{F'}{\FR}\Big)\cap \Big(\frac{F'}{\FR}\Big)'.\] Clearly \[\frac{F'\cap R}{\FR}\subseteq Z\Big(\frac{F'}{\FR}\Big)\] and \[\frac{F'\cap R}{\FR} \subseteq \frac{F'}{\FR}.\] It thus suffices to show that \[\Big(\frac{F'}{\FR}\Big)' = \frac{F'}{\FR}.\] We first note that \[\Big(\frac{F'}{\FR}\Big)' = \frac{F''+\FR}{\FR}\] and that $F''+\FR \subseteq F'$. Since $L=L'$, we know $F/R = (F/R)' = \frac{F'+R}{R}$. This implies that for all $x_i\in F$, $x_i=y_i+r_i$ for some $y_i\in F'$ and $r_i\in R$. Therefore \begin{align*}
    x_1\dashv x_2 &= y_1\dashv y_2 + y_1\dashv r_2 + r_1\dashv y_2 + r_1\dashv r_2
\end{align*} and \begin{align*}
    x_1\vdash x_2 &= y_1\vdash y_2 + y_1\vdash r_2 + r_1\vdash y_2 + r_1\vdash r_2,
\end{align*} both of which are elements of $F''+\FR$. This implies that $F'\subseteq F''+\FR$, and thus $F'/(\FR)$ is a cover of $L$.
\end{proof}

\begin{cor}
Let $L$ be a finite-dimensional perfect diassociative algebra with free presentation $0\xrightarrow{} R\xrightarrow{} F\xrightarrow{} L\xrightarrow{} 0$. Then $F'/(\FR)$ is perfect.
\end{cor}

\begin{proof}
By the proof of Theorem \ref{batten 6.8}, we have $\big(F'/(\FR)\big)' = F'/(\FR)$.
\end{proof}

\begin{thm}\label{batten 6.10}
Let $L$ be a finite-dimensional perfect diassociative algebra with free presentation $0\xrightarrow{} R\xrightarrow{} F\xrightarrow{} L\xrightarrow{} 0$. Then \[E:0\xrightarrow{} \frac{F'\cap R}{\FR} \xrightarrow{} \frac{F'}{\FR} \xrightarrow{} L\xrightarrow{} 0\] is universal.
\end{thm}

\begin{proof}
Let $E_1:0\xrightarrow{} A\xrightarrow{} H\xrightarrow{} L\xrightarrow{} 0$ be a central extension of $L$. By Lemma 3.3 of \cite{mainellis batten di}, it is covered by a natural exact sequence, call it $E_2$, making the diagram \[\begin{tikzcd}
E_2:0\arrow[r] &\frac{R}{\FR}\arrow[r] \arrow[,d] & \frac{F}{\FR} \arrow[d,swap,"\beta"] \arrow[r] & L\arrow[r] \arrow["\text{id}",d]&0\\
E_1:0\arrow[r] & A\arrow[r,swap] & H \arrow[r,swap] &L\arrow[r] &0
\end{tikzcd}\] commute. Since \[\begin{tikzcd}
E:0\arrow[r] &\frac{F'\cap R}{\FR}\arrow[r] \arrow[,d] & \frac{F'}{\FR} \arrow[d,swap,"\theta"] \arrow[r] & L\arrow[r] \arrow["\text{id}",d]&0\\
E_1:0\arrow[r] & A\arrow[r,swap] & H \arrow[r,swap] &L\arrow[r] &0
\end{tikzcd}\] commutes, where $\theta = \beta|_{F'/(\FR)}$, $E_1$ is covered by $E$. Since $F'/(\FR)$ is perfect and $E$ covers $E_1$, Lemma \ref{batten 6.3} implies that $E$ uniquely covers $E_1$. Thus $E$ is universal.
\end{proof}

\begin{thm}\label{batten 6.11}
If $0\xrightarrow{} A\xrightarrow{} L^*\xrightarrow{} L\xrightarrow{} 0$ is a universal central extension and $L$ is perfect, then $A\cong M(L)$ and $L^*$ is a cover $L$.
\end{thm}

\begin{proof}
We know that \[0\xrightarrow{} \frac{F'\cap R}{\FR} \xrightarrow{} \frac{F'}{\FR} \xrightarrow{} L\xrightarrow{} 0\] is universal. By Lemma \ref{batten 6.1}, there exists an isomorphism \[L^*\xrightarrow{} \frac{F'}{\FR}\] which carries $A$ onto \[\frac{F'\cap R}{\FR}\cong M(L).\] Thus $A\cong M(L)$ and $L^*$ is a cover of $L$.
\end{proof}

\begin{thm}\label{batten 6.12}
Let $L$ be a finite-dimensional perfect diassociative algebra and let $M(L)=0$. Then $\cH^2(L,A)=0$ for any central module $A$ of $L$.
\end{thm}

\begin{proof}
Since $M(L)=0$, Theorem \ref{batten 6.10} implies that the extension \[0\xrightarrow{} 0\xrightarrow{} \frac{F'}{\FR}\xrightarrow{} L\xrightarrow{} 0\] is universal. By Lemma \ref{batten 6.7}, the extension $0\xrightarrow{} 0\xrightarrow{} L\xrightarrow{} L\xrightarrow{} 0$ is also universal. By Lemma \ref{batten 6.4}, every central extension of $L$ splits, and thus $\cH^2(L,A) = 0$.
\end{proof}

\begin{thm}
Let $L$ be a finite-dimensional perfect diassociative algebra and let $M(L)=0$. If $Z$ is a central ideal of $L$, then $Z\cong M(L/Z)$ and $L$ is a cover of $L/Z$.
\end{thm}

\begin{proof}
By the proof of Theorem \ref{batten 6.12}, the extension $0\xrightarrow{} 0\xrightarrow{} L\xrightarrow{} L\xrightarrow{} 0$ is universal. Since $L$ is perfect, Lemma \ref{batten 6.6} implies that $0\xrightarrow{} Z\xrightarrow{} L\xrightarrow{} L/Z\xrightarrow{} 0$ is also universal. Thus, by Theorem \ref{batten 6.11}, $M(L/Z)\cong Z$, which implies that $L$ is the cover of $L/Z$.
\end{proof}

\begin{thm}
Let $L$ be a finite-dimensional perfect diassociative algebra and $C$ be a cover of $L$. Then $C=C'$ and $M(C) = 0$.
\end{thm}

\begin{proof}
Let $A=M(L)$. Then $L\cong C/A$ and $A\subseteq Z(C)\cap C'$. One computes \[L/L'\cong \frac{C/A}{(C/A)'} \cong \frac{C/A}{(C'+A)/A} \cong \frac{C}{C'+A}.\] Since $L=L'$ and $A\subseteq C'$, we have $C=C'+A = C'$. Thus $C$ is perfect. We now invoke an extended Hochschild-Serre type spectral sequence that was obtained in \cite{mainellis batten di}. Letting $\F$ denote our ground field, the sequence \[0\xrightarrow{} \Hom(L,\F)\xrightarrow{\Inf_1} \Hom(C,\F)\xrightarrow{\Res} \Hom(A,\F)\xrightarrow{\Tra} M(L)\xrightarrow{\Inf_2} M(C) \xrightarrow{\delta} (C/C'\otimes A\oplus A\otimes C/C')^2\] is exact. Here, the term $(C/C'\otimes A\oplus A\otimes C/C')^2$ must be zero since $C=C'$, which yields $M(C) = \ker \delta = \ima(\Inf_2)$. Next, we also know that $\Hom(C,\F)=0$ since $C$ is perfect. This implies that $0=\ima(\Res) = \ker(\Tra)$. Then $\ima(\Tra) \cong \Hom(A,\F) \cong A\cong M(L)$ and therefore $M(C) = \ima(\Inf_2) \cong M(L)/\ker(\Inf_2) = M(L)/\ima(\Tra) = M(L)/M(L) \cong 0$.
\end{proof}

\section*{Acknowledgements}
The author would like to thank Ernest Stitzinger for the many helpful discussions.

\end{document}